\documentclass[12pt]{article}
\usepackage[utf8]{inputenc}
\usepackage[T1]{fontenc}
\usepackage[english]{babel}
\usepackage{indentfirst}
\usepackage{latexsym}
\usepackage[colorlinks=true, bookmarksnumbered=true, bookmarksopen=true,
bookmarksopenlevel=3, pdfstartview=FitH, linkcolor=cyan, pdfmenubar=true,
pdftoolbar=true, bookmarks=true,citecolor=cyan, urlcolor=magenta,
filecolor=cyan,menucolor=black,plainpages=false,pdfpagelabels]{hyperref}
\usepackage[lmargin=2cm,tmargin=2cm,bmargin=2cm,rmargin=2cm]{geometry}
\usepackage{amsbsy, amsmath,amsfonts,amssymb,amsthm}
\usepackage{times}
\usepackage{graphicx}
\usepackage{amsmath}
\usepackage{amsfonts}
\usepackage{amssymb}
\usepackage[all]{xy}
\newcommand{\N}{\mathbb{N}}
\newcommand{\C}{\mathbb{C}}

\newtheorem{theorem}{Theorem}[section]
\newtheorem{definition}[theorem]{Definition}

\newtheorem{lemma}[theorem]{Lemma}
\newtheorem{remark}[theorem]{Remark}

\numberwithin{equation}{section}

\begin{document}
	\title{Chaos in convolution operators on the space of entire functions of infinitely many complex variables}
	\author{Blas M. Caraballo \thanks{The first named author is supported by CAPES and CNPq } and Vin\'icius V. F\'avaro\thanks{The second named author is supported by FAPEMIG Grant APQ-03181-16; and CNPq Grant 310500/2017-6.}}
	
	\date{}
	\maketitle
	
	\vspace{-0.2 cm}
	
	\begin{center}
		\textit{Dedicated to the memory of Professor Jorge Mujica (1946-2017)}
	\end{center}
	
	\vspace{0.2 cm}
	
	\begin{abstract}
		A classical result of Godefroy and Shapiro states that every nontrivial convolution operator on the space $\mathcal{H}(\mathbb{C}^n)$ of entire functions of several complex variables is hypercyclic. In sharp contrast with this result F\'avaro and Mujica show that no translation operator on the space $\mathcal{H}(\mathbb{C}^\mathbb{N})$ of entire functions of infinitely many complex variables is hypercyclic. In this work we study the linear dynamics of convolution operators on $\mathcal{H}(\mathbb{C}^\mathbb{N})$. First we show that no convolution operator on $\mathcal{H}(\mathbb{C}^\mathbb{N})$ is neither cyclic nor $n$-supercyclic for any positive integer $n$. After we study the notion of Li--Yorke chaos in non-metrizable topological vector spaces and we show that  every nontrivial convolution operator on $\mathcal{H}(\mathbb{C}^\mathbb{N})$ is Li--Yorke chaotic.
		
		%{\small \bigskip\noindent\textbf{MSC2010:}  46E50, 46B28, 46A32}
		{\small \bigskip\noindent\textbf{MSC2010:}  47A16, 47B38, 32A15}
		
		{\small \medskip\noindent\textbf{Keywords:}   $n$-supercyclicity, cyclicity, Li-Yorke chaos, convolution operators, holomorphic functions of infinitely many complex variables.}
	\end{abstract}
	
	\section{Introduction}
	Let $V$ be a subset of a Hausdorff topological complex vector space $E$ and let $T\colon E\to E$ be a continuous linear operator (from now on we just write operator). The \emph{orbit of $V$ under $T$}, denoted by orb$_T(V)$, is the subset of $E$ given by 
	\[\textnormal{orb}_T(V)=\bigcup_{k=0}^\infty T^k(V).\]
	If $V=\{x\}$ is a singleton and orb$_T(V)=\{T^kx : k\in\N_0\}$ is dense in $E$, where $\N_0=\{0,1,2,3,\ldots\}$, then $T$
	is said to be \emph{hypercyclic} and $x$ a \emph{hypercyclic vector} for $T$. If the linear space generated by orb$_T(V)$ is dense in $E$, then $T$
	is said to be \emph{cyclic}	and $x$ a \emph{cyclic vector} for $T$. If $V=\textnormal{span}\{x\}$ and orb$_T(V)=\C\cdot\{T^kx : k\in\N_0\}$ is dense in $E$, then $T$ is said to be \emph{supercyclic} and $x$ a \emph{supercyclic vector} for $T$. Finally, if $V$ is a vector subspace of dimension $n$ and orb$_T(V)$ is dense in $E$, then $T$
	is said to be \emph{$n$-supercyclic} and $V$ a \emph{supercyclic subspace} for $T$.
	
	Hypercyclicity is the most important concept in linear dynamics and it has received considerable attention in the last 25 years.  References \cite{bay, grosse2011linear} provide deep and detailed surveys of the theory. The notion of chaos in linear dynamics was introduced by Godefroy and Shapiro \cite{godefroy} in 1991. They adopted the Devaney's definition of chaos. Recall that an operator on a Fréchet space is \emph{chaotic} if it is hypercyclic and it has a dense set of periodic points. 
	
	There are several important notions of chaos and some authors have started to study this notions in the context of linear dynamics. In addition to the notions defined above we mention the first mathematical definition of chaos given in 1975 by Li and Yorke in  \cite{li1975period}, which is currently known as Li--Yorke chaos. This classical notion of Li-Yorke chaos was introduced for maps defined on metric spaces as follow: Given a metric space $(M,d)$ and a continuous map $f\colon M\rightarrow M$, we recall that a pair $(x,y)\in M\times M$ is called a \emph{Li-Yorke pair} for $f$ if
	$$
	\liminf_{n\rightarrow\infty} d(f^n(x),f^n(y))=0 \textrm{ and } \limsup_{n\rightarrow\infty}  d(f^n(x),f^n(y))>0.$$
	A \emph{scrambled set} for $f$ is a subset $S$ of $M$ such that $(x, y)$ is a Li–Yorke pair for $f$ whenever $x$ and $y$ are distinct points in $S$. The map $f$ is said to be \emph{Li–Yorke chaotic} if there exists an uncountable scrambled set for $f$. By \cite[Theorem 9]{Li-Yorke} hypercyclicity implies chaos in the sense of Li--Yorke. 
	
	In this paper we are mainly interested in the linear dynamics of convolution operators on spaces of entire functions of infinitely many complex variables. We remark that several results on linear dynamics of operators on spaces of entire functions of infinitely many complex variables have appeared in the last few decades. See for instance \cite{aronbes, bay2, BBFJ, bes2012, CDSjmaa, vinicius2016hypercyclic, FMNach, GS, goswinBAMS, MPS, peterssonjmaa}.
	
	A classical result due to Godefroy and Shapiro \cite{godefroy}  states that every nontrivial convolution o\-pe\-ra\-tor on $\mathcal{H}(\C^n)$ is hypercyclic. Moreover,  A. Bonilla and K.-G. Grosse-Erdmann \cite{bonilla2006theorem} showed that these convolution operators are even frequently hypercyclic, which is a stronger notion than hypercyclicity. In sharp contrast with these results, Fávaro and Mujica \cite{vinicius2016hypercyclic} proved that no convolution operator on $\mathcal{H}(\C^\N)$  can be hypercyclic. At first sight this result may look surprising, since it is well known that every $f\in\mathcal{H}(\mathbb{C}^\mathbb{N})$ depends only of  finitely many variables (see \cite[p. 162]{dineen2012complex}). Based on these facts, the following question arises:
	\begin{center}
		Do the convolution operators on $\mathcal{H}(\mathbb{C}^\mathbb{N})$ satisfy some notion of the linear dynamics weaker than hypercyclicity?
		%Do the convolution operators on $\mathcal{H}(\mathbb{C}^\mathbb{N})$ satisfy some weaker notion of chaos than hypercyclicity?
	\end{center}
	
	Note that the notions of $1$-supercyclicity and supercyclicity are equivalent and that the following diagram holds:
	\[\xymatrix@=7mm{\textnormal{hypercyclicity} \ar@2{->}[r] & \textnormal{supercyclicity}\ar@2{->}[r]\ar@2{->}[d]&\textnormal{cyclicity}\\
		&n\textnormal{-supercyclicity}\ar@2{}[ru]}\]
	
	An $n$-supercyclic operator need not be cyclic, for $n=2,3,\ldots$ (for an example in infinite dimension see \cite{bourdon2004}). Hilden and Wallen \cite{hilden1974some} proved that no operator on $\C^n$ can be supercyclic $(n=2,3\ldots)$. So, $n$-supercyclicity does not imply supercyclicity, in general. For properties and results about supercyclicity and $n$-supercyclicity we refer to \cite{bourdon2004,feldman2002n,herrero1991limits,hilden1974some}. 
	
	In sharp contrast with the aforementioned result of Godefroy and Shapiro we will show that no convolution operator on $\mathcal{H}(\C^\N)$  can be neither cyclic nor $n$-supercyclic for any positive integer $n$ (Theorem \ref{nsuper}). 
	So we may rewrite the last question in the following way:
	
	\begin{center}
		Are the convolution operators on $\mathcal{H}(\mathbb{C}^\mathbb{N})$ at least Li--Yorke chaotic?
	\end{center}
	
	Since $\mathcal{H}(\C^\N)$ is a non-metrizable complete locally convex space, the classical notion of Li--Yorke chaos does not make sense in this context. Recently T. Arai \cite{arai2017devaney} introduced the notion of Li-Yorke chaos for an action of a group on an uniform space. Since every topological vector space is an uniform space, we will adopt the Arai's definition of Li-Yorke chaos. Using this definition we will prove that the last question has positive answer, i.e., every convolution operator on $\mathcal{H}(\mathbb{C}^\mathbb{N})$  is Li--Yorke chaotic (Theorem \ref{TLY}). It is worth to mention that the criteria that appear in the literature to prove that an operator satisfies (or not) some kind of linear chaos are in general for operators defined on $F$-spaces. Since $\mathcal{H}(\C^\N)$ is not a metric space we will use no criterion to prove the first result (Theorem \ref{nsuper}). However, to show the second result (Theorem \ref{TLY}) we will adapt a criterion obtained by Bernardes \textit{et al} \cite{Li-Yorke} for operators on Fréchet spaces to operators on Hausdorff topological vector spaces. This criterion is the key of the proof.
	
	%Since $\mathcal{H}(\C^\N)$ is not a metric space we will use no criterion to prove the first result and we will adapt a criteria obtained by Bernardes \textit{et al} \cite{Li-Yorke} for operators on Fréchet spaces to operators on Hausdorff locally convex spaces.
	
	For our purpose it is enough to present the definition of Li--Yorke chaos for an operator $T$ on a Hausdorff topological vector space $E$ as follow: A pair $(x,y)\in E\times E$ is said to be \emph{asymptotic} for $T$ if for any neighborhood of zero $U$, there exists $k\in\N$ such that $T^n(x-y)\in U$ for every $n\geq k$, that is, if $T^n(x-y)\to0.$
	A pair $(x,y)\in E\times E$ is said to be \emph{proximal} for $T$ if for any neighborhood of zero  $U$, there exists $n\in\N$ such that $T^n(x-y)\in U$, that is, if the sequence $\{T^n(x-y)\}$ has a subsequence converging to zero.
	
	A pair $(x,y)\in E\times E$ is said to be a \emph{Li--Yorke pair} for $T$ if it is proximal, but it is not asymptotic. In other words, $(x,y)$ is a Li--Yorke pair for $T$ if and only if the sequence $\{T^n(x-y)\}$ does not converge to zero, but it has a subsequence converging to zero.
	
	A \emph{scrambled set} for $T$ is a subset $S$ of $E$ such that $(x,y)$ is a Li--Yorke pair for $T$ whenever $x$ and $y$ are distinct points in $S$. Finally, we say that  $T$ is \emph{Li--Yorke chaotic} if there exists an uncountable scrambled set for $T$. 
	
	It is easy to check that if $E$ is metrizable and we consider a translation-invariant metric (this metric exists by definition of metrizability), then both definitions of Li--Yorke chaos coincide.
	
	\section{Preliminaries}
	
	Given the topological product $\C^{\N}=\prod_{n=1}^{\infty}\C$, we consider the complex vector space of all entire functions $f\colon\C^{\N}\to\C$, which is denoted by $\mathcal{H}(\C^\N)$. It is well known that there are only two usual locally convex topologies on $\mathcal{H}(\C^\N)$: the compact open topology $\tau_0$ and its bornological associated topology $\tau_\delta$ (see \cite{Barroso1971,dineen2012complex}). It is also known that, with both topologies, $\mathcal{H}(\C^\N)$ is separable. For details and properties of these topologies we refer to \cite{ansemil1979topological,Barroso1971,barroso1979some}.
	
	For each $n\in\N$ we consider the canonical inclusion $J_n\colon\C^n\to\C^\N$, the canonical projection $\pi_n\colon\C^\N\to\C^n$ and the corresponding linear applications 
	\[J_n^*\colon f\in\mathcal{H}(\C^\N)\to f\circ J_n\in\mathcal{H}(\C^n), \ \hspace{1cm} \ \pi_n^*\colon f_n\in\mathcal{H}(\C^n)\to f_n\circ \pi_n\in\mathcal{H}(\C^\N).\]
	Since $\pi_n\circ J_n=Id_{\C^n}$ it follows that 
	\begin{equation}\label{complemented}
	J_n^*\circ\pi_n^*=Id_{\mathcal{H}(\C^n)}, \textrm{ for each } n\in\N.
	\end{equation} 
	So $\mathcal{H}(\C^n)$ can be seen as the vector subspace of $\mathcal{H}(\C^\N)$ of all  entire functions on $\C^\N$ that depend only of the $n$ first variables, through the injective application $\pi_n^*$, for each $n\in\N$. It is easy to check that 
	\begin{equation}
	\pi_1^*(\mathcal{H}(\C))\subset\pi_2^*(\mathcal{H}(\C^2))\subset\cdots\subset\pi_n^*(\mathcal{H}(\C^n))\subset\cdots\subset\mathcal{H}(\C^\N).\label{inclusao}
	\end{equation}
	
	By \cite[p. 162]{dineen2012complex} or \cite[Corolário 38]{Barroso1971}
	\begin{equation}\label{barro}
	\mathcal{H}(\C^\N)=\bigcup_{n=1}^\infty\{f_n\circ\pi_n\colon f_n\in\mathcal{H}(\C^n)\}=\bigcup_{n=1}^\infty\pi_n^*(\mathcal{H}(\C^n)).
	\end{equation}
	Also, by \cite[Proposition 1.3]{ansemil1979topological} the topology $\tau_\delta$, which was independently introduced by Nachbin \cite{nachbin1970espaces} and Couré \cite{coeure1970fonctions} coincides with the inductive limit topology of the Fréchet spaces $\mathcal{H}(\C^n)$, $n\in\N$, that is, 
	\[(\mathcal{H}(\C^\N),\tau_\delta)=\textnormal{ind}_{n\in\N}\mathcal{H}(\C^n),\]
	where $\mathcal{H}(\C^n)$ is endowed with its usual topology, the compact open topology. More precisely, $\tau_\delta$ is the strongest locally convex topology on $\mathcal{H}(\C^\N)$, which becomes the applications $\pi_n^*$ continuous. If $\tau$ represents any of the topologies $\tau_0$, $\tau_\delta$ on $\mathcal{H}(\C^\N)$ then the linear operators
	\[J_n^*\colon f\in(\mathcal{H}(\C^\N),\tau)\to f\circ J_n\in\mathcal{H}(\C^n) \ \hspace{1cm} \ \pi_n^*\colon f_n\in\mathcal{H}(\C^n)\to f_n\circ \pi_n\in(\mathcal{H}(\C^\N),\tau)\]
	are continuous and it follows from \eqref{complemented} that $\mathcal{H}(\C^n)$  is topologically isomorphic to a complemented subspace of $(\mathcal{H}(\C^\N),\tau)$. In particular, $\pi_n^*(\mathcal{H}(\C^n))$ is a closed proper subspace of $(\mathcal{H}(\C^\N),\tau)$. For background information on these topologies we refer the reader to the book of Dineen  \cite{dineen2012complex}.
	
	Finally we recall that the \emph{translation operator by $\xi \in \C^\N$}, $$\tau_{\xi}\colon \mathcal{H}(\C^\N) \rightarrow \mathcal{H}(\C^\N) $$ is given by $(\tau_{\xi}f)(x) = f(x-\xi)$ for every $x \in \C^\N$. Analogously we define translation operators on $\mathcal{H}(\C^n)$ for each $n\in\N.$
	
	\begin{remark}\label{inclu}\rm
		It is interesting to note that, if $\xi\in\C^\N$ is such that $\pi_n(\xi)=0$, then the translation operator $\tau_\xi$ on $\mathcal{H}(\C^\N)$ coincides with the identity operator on $\pi_n^*(\mathcal{H}(\C^n))$, that is, $\tau_\xi|_{\pi_n^*(\mathcal{H}(\C^n))}=Id$. Hence ${\pi_n^*(\mathcal{H}(\C^n))}$ is a closed proper subspace of $\mathcal{H}(\C^\N)$ and $\tau_\xi$-invariant.   
	\end{remark}
	
	\subsection{Convolution operators on $\mathcal{H}(\C^\N)$}
	
	In this section we prove some technical results about convolution operators that we need to show the main results of this work.
	
	\begin{definition}\rm
		A \textit{convolution operator} on $\mathcal{H}(\C^\N)$ 
		is a continuous linear mapping 
		$$ L\colon \mathcal{H}(\C^\N) \rightarrow \mathcal{H}(\C^\N) $$ 
		such that $L(\tau_{\xi}f) = \tau_{\xi}(Lf)$ for every $f \in \mathcal{H}(\C^\N)$ and $\xi \in \C^\N$. Analogously we define convolution operators on $\mathcal{H}(\C^n)$ for each $n\in\N.$
	\end{definition}
	
	%	In this section we will write the orbit of a function $f\in\mathcal{H}(\C^\N)$ under a convolution operator of different two way. The proof of the following lemma is similar to the proof of \cite[Theorem 4.1(ii)]{vinicius2016hypercyclic}.
	\begin{lemma}\label{Lm1}
		Let $f\in\mathcal{H}(\C^\N)$, $k\in\N$ and $T$ be a linear operator on $\mathcal{H}(\C^\N)$. Then there is $N\in\N$ such that  $T^if=\pi_{N}^*((T^if)\circ J_N)$ for every $i=0,\ldots,k$. Moreover, if $f=f_n\circ\pi_n$ with $f_n\in\mathcal{H}(\C^n)$ for some $n\in\N$, and $L$ is a linear operator on $\mathcal{H}(\C^\N)$ that commutes with all the translation operators (in particular, if $L$ is a convolution operator), then
		\[L^kf=\pi_{n}^*((L^kf)\circ J_n)=\pi_n^*\circ J_{n}^*(L^kf)\] 
		for every $k\in\N_0$.
	\end{lemma}
	\begin{proof}
		By \eqref{barro} we may write
		%	It is clear that we may write 
		$$f=f_{n_0}\circ\pi_{n_0}, Tf=f_{n_1}\circ\pi_{n_1},\ldots, T^kf=f_{n_k}\circ\pi_{n_k},$$ with $f_{n_i}\in\mathcal{H}(\C^{n_i})$ for every $i=0,\cdots,k$. Let $N=\max\{n_i: i=0,\ldots, k\}$ and $\xi\in\C^\N$ be such that $\pi_N(\xi)=0$. Then $\pi_{n_i}(\xi)=0$ for every $i=0,\cdots,k$ and
		\[\tau_\xi(T^{i}f)(x)=(T^{i}f)(x-\xi)=f_{n_i}\circ\pi_{n_i}(x-\xi)=f_{n_i}(\pi_{n_i}(x)-\pi_{n_i}(\xi))=f_{n_i}(\pi_{n_i}(x))=(T^{i}f)(x)\]
		for every $x\in\C^\N$ and every $i=0,\cdots,k$. On the other hand, given any $x=(x_j)\in\C^\N$, if we take $\xi=(0,\ldots,0,x_{N+1},x_{N+2},\ldots)\in\C^\N$, then
		\begin{align*}
		(T^if)(x)&=\tau_\xi(T^if)(x)=(T^if)(x-\xi)=(T^if)(x_1,\ldots,x_N,0,0,\ldots)\\
		&=(T^if)(J_N\circ\pi_N(x))=(T^if)\circ J_N\circ\pi_N(x).
		\end{align*}
		Thus $T^if=\pi_N^*((T^if)\circ J_N)$ for every $i=0,\ldots,k$.
		
		Now, suppose that $f=f_n\circ\pi_n$ and $L$ is a linear operator on $\mathcal{H}(\C^\N)$ that commutes with all the translation operators. Choosing $\xi\in\C^\N$ such that $\pi_n(\xi)=0$ we get
		\begin{align*}
		\tau_\xi(f)=f, \hspace{0.3cm} \tau_\xi(Lf)=L(\tau_\xi f)=Lf, \dots,
		\tau_\xi(L^kf)=L(\tau_\xi(L^{k-1}f))=L^kf,
		\end{align*}
		for every $k\in\N_0$. Following the same lines of the first part of the proof we obtain 
		\[L^kf=\pi_{n}^*((T^kf)\circ J_n),\] 
		for every $k\in\N_0$.
	\end{proof}
	%The lemma we just prove
	This lemma tells us that the operator $\pi_n^*\circ J_{n}^*$ acts as the identity on $\textnormal{Orb}_L(f)$, whenever $f=f_n\circ\pi_n\in\mathcal{H}(\C^\N)$ and $L$ is a convolution operator.	
	
	If $f=f_n\circ\pi_n\in\mathcal{H}(\C^\N)$ with $f_n\in\mathcal{H}(\C^n)$, and $\xi\in\C^\N$ is not difficult to verify that $\tau_\xi f=(\tau_{\pi_n(\xi)}f_n)\circ\pi_n$. In this sense the following question is quite natural:
	
	\medskip	
	
	\textit{
		Does every convolution operator $L$ on $(\mathcal{H}(\C^\N),\tau)$ satisfy $Lf=L_nf_n\circ\pi_n$, where $L_n$ is a convolution operator on $\mathcal{H}(\C^n)$?
	}
	
	\medskip
	
	The following lemma gives a positive answer to this question.
	
	\begin{lemma}\label{GJD}
		Let $L$ be a convolution operator on $(\mathcal{H}(\C^\N),\tau)$.
		\begin{enumerate}
			\item [(a)] The operator $L_n:=J_n^*\circ
			L\circ\pi_n^*\colon\mathcal{H}(\C^n)\to \mathcal{H}(\C^n)$, $n\in\N$, is a convolution operator. We say that $L_n$ is \emph{the convolution operator on $\mathcal{H}(\C^n)$ associated to $L$}.  %called \emph{convolution operator on $\mathcal{H}(\C^n)$ associated to $L$}.
			\item[(b)] \[L(f_n\circ\pi_n)=(L_nf_n)\circ\pi_n, \ \ \textnormal{for every} \ \ f_n\in\mathcal{H}(\C^n) \ \textnormal{and} \ n\in\N.\]
			\item[(c)] $L$ is a scalar multiple of the identity on $\mathcal{H}(\C^\N)$ if and only if $L_n$ is a scalar multiple of the identity on $\mathcal{H}(\C^n)$, for every $n\in\N$.
		\end{enumerate}
	\end{lemma}
	\begin{proof}
		Let $n\in\N$, $f_n\in\mathcal{H}(\C^n)$ and $ f:=f_n\circ\pi_n\in\mathcal{H}(\C^\N)$.
		\begin{enumerate}
			\item [(a)]  Note that 
			\begin{equation}\label{GJ}
			L_nf_n=J_n^*\circ L\circ\pi_n^*(f_n)=J_n^*\circ L(f)=(Lf)\circ J_n.
			\end{equation}
			Let $a\in\C^n$. We want to show that $\tau_a\circ L_n=L_n\circ\tau_a$. Applying \eqref{GJ} we have
			\begin{align*}
			[\tau_a\circ L_n](f_n)(z)&=\tau_a(L_nf_n)(z)=(L_nf_n)(z-a)=(Lf)\circ J_n(z-a)=(Lf)(J_n(z)-J_n(a))\\
			&=[\tau_{J_n(a)}(Lf)](J_n(z)),
			\end{align*}
			for every $z\in\C^n$ and so $[\tau_a\circ L_n](f_n)=[\tau_{J_n(a)}(Lf)]\circ J_n$. Using the fact that $L$ is a convolution operator we get
			\begin{align*}
			[\tau_a\circ L_n](f_n)&=[L(\tau_{J_n(a)}f)]\circ J_n=[L((\tau_af_n)\circ \pi_n)]\circ J_n=[L\circ \pi_n^*(\tau_af_n)]\circ J_n\\&=J_n^*\circ L\circ \pi_n^*(\tau_af_n)
			=[L_n\circ\tau_a](f_n).
			\end{align*}
			
			\item [(b)] Applying Lemma \ref{Lm1} to the entire function $f_n\circ\pi_n$ and using \eqref{GJ} we get
			\begin{equation}\label{GJchave}
			L(f_n\circ\pi_n)=\pi_n^*((L(f_n\circ\pi_n)) \circ J_n)=\pi_n^*(L_nf_n)=(L_nf_n)\circ\pi_n.
			\end{equation}
			\item[(c)] Let $\lambda\in\C$ be such that $Lg=\lambda g$ for every $g\in\mathcal{H}(\C^\N)$. Then 
			\[(L_nf_n)\circ\pi_n=Lf=(\lambda f_n)\circ\pi_n.\]
			Since $\pi_n^*$ is injective, it follows that $L_nf_n=\lambda f_n$. Therefore $L_n$ is a scalar multiple of the identity. 
			
			Conversely, suppose that for each $n\in\N$ there exists $\lambda_n\in\C$ such that $L_nf_n=\lambda_nf_n$ for every $f_n\in\mathcal{H}(\C^n)$. It is not difficult to verify that, if $g\in\pi_n^*(\mathcal{H}(\C^n))$, then $Lg=\lambda_ng$. Note that to prove the assertion it suffices to show that $\lambda_n=\lambda_m$ for any $n,m\in\N$. So, let $n,m\in\N$ with $n\leq m$. By (\ref{inclusao}) we may choose $g\in\mathcal{H}(\C^\N)$ such that $g\neq0$ and $g\in\pi_n^*(\mathcal{H}(\C^n))\subset\pi_m^*(\mathcal{H}(\C^m))$. Thus $\lambda_ng=Lg=\lambda_mg$ and since $g\neq0$ it follows that $\lambda_n=\lambda_m$.		\end{enumerate}
	\end{proof}
	
	Below we list some remarks about the previous lemma.	
	
	\begin{remark} 
		\begin{enumerate}
			\item[$(1)$] For $\xi\in\C^\N$, the convolution operator  $(\tau_\xi)_n=\tau_{\pi_n(\xi)}$ is a concrete example of convolution operator which is associated to the translation $\tau_{\xi}$.	
			\item [$(2)$] The fact that $L$ commutes with all the translation operators is very important to show Lemma \ref{GJD}(b). In fact, if $L$ is a linear operator on $\mathcal{H}(\C^\N)$ that does not commute  with all the translation operators and $f=f_n\circ\pi_n\in\mathcal{H}(\C^\N)$ for some $n\in\N$, then it follows from Lemma \ref{Lm1} that there exists $N\in\N$, $N\geq n$ (not necessarily equal), such that $f=f_N\circ\pi_N$ and $Lf=\pi_N^*((Lf)\circ J_N)$. Now, using \eqref{GJ} we have $Lf=(L_Nf_N)\circ\pi_N$. Hence, we may not ensure that $L$ can be factored in the form $Lf=(L_nf_n)\circ\pi_n$.
			\item [$(3)$] Lemma \ref{GJD}(b) allows us to write the orbit  Orb$_L(f)$ in terms of convolution operators on $\mathcal{H}(\C^n)$, that is, if $f=f_n\circ\pi_n\in\mathcal{H}(\C^\N)$ and $L$ is a convolution operator on $(\mathcal{H}(\C^\N),\tau)$, then
			\[Lf=(L_nf_n)\circ\pi_n, \hspace{0,3cm}L^2f=L((L_nf_n)\circ\pi_n)=L_n(L_nf_n)\circ\pi_n=(L_n^2f_n)\circ\pi_n,\]
			and proceeding by induction it follows that
			\[L^kf=(L_n^kf_n)\circ\pi_n,\]
			for every $k\in\N_0$.
		\end{enumerate}
	\end{remark}
	
	\section{Linear dynamics of convolution operators on $\mathcal{H}(\C^\N)$}
	In this section we will study the linear dynamics of convolution operators on $\mathcal{H}(\C^\N)$. We start by proving that convolution operators on $\mathcal{H}(\mathbb{C}^{\mathbb{N}})$ are neither cyclic nor $n$-supercyclic for any $n\in\N$. This result improves a result of Fávaro and Mujica \cite{vinicius2016hypercyclic}, which states that no convolution operator on $\mathcal{H}(\C^\N)$ is hypercyclic.
	\begin{theorem}\label{nsuper}
		\begin{enumerate}
			\item [(a)] No convolution operator on $\mathcal{H}(\C^\N)$ is cyclic.
			\item [(b)] No convolution operator on $\mathcal{H}(\C^\N)$ is $n$-supercyclic, for any  $n\in\N$. %choice of
		\end{enumerate}
	\end{theorem}
	
	\begin{proof}
		Let $L$ be a convolution operator on $\mathcal{H}(\C^\N)$.  
		\begin{enumerate}
			\item [(a)] Let $f\in\mathcal{H}(\C^\N)$. Then we may write  $f=f_n\circ\pi_n$ with $f_n\in\mathcal{H}(\C^n)$. By Lemma \ref{Lm1} the orbit of $f$ under $L$ is
			\[\textnormal{Orb}_L(f)=\{L^kf : k\in\N_0\}=\{\pi_{n}^*((L^kf)\circ J_n) : k\in\N_0\}\subset\pi_n^*(\mathcal{H}(\C^n)).\] 
			Since $\pi_n^*(\mathcal{H}(\C^n))$ is a closed  proper subspace of $(\mathcal{H}(\C^\N),\tau)$, we have 
			\[\overline{\textrm{span } \textnormal{Orb}_L(f)}^\tau\subset\pi_n^*(\mathcal{H}(\C^n)).\]
			Therefore $\textrm{span } \textnormal{Orb}_L(f)$  cannot be a dense subset of $(\mathcal{H}(\C^\N),\tau)$. So there is not a cyclic entire function for any convolution operator on $(\mathcal{H}(\C^\N),\tau)$. Hence $L$ is not cyclic.
			
			\item [(b)] Let $n\in\N, n>1$,  and let $V$ be an $n$-dimensional vector subspace of $\mathcal{H}(\C^\N)$ with generators $f_1,\ldots,f_n$. Then $L^k(V)$ is a vector subspace of $\mathcal{H}(\C^\N),$ generated by  $L^kf_1,\cdots,L^kf_n$, and with dimension less than or equal to $n$, for every $k\in\N_0$. If we write $f_1=f_{m_1}\circ\pi_{m_1}, \cdots,f_n=f_{m_n}\circ\pi_{m_n}$, with $f_{m_i}\in\mathcal{H}(\C^{m_i})$, for every $i=1,\cdots,n$, then it follows from Lemma \ref{Lm1} and (\ref{inclusao}) that
			\[L^k(V)\subset\pi_{m_1}^*(\mathcal{H}(\C^{m_1}))+\cdots+\pi_{m_n}^*(\mathcal{H}(\C^{m_n}))\subset \pi_{m}^*(\mathcal{H}(\C^{m})),\]
			for every  $k\in\N_0$, where $m:=\max\{m_i:  i=1,\dots, n\}$. Therefore
			\[\textnormal{orb}_L(V)=\bigcup_{k=0}^\infty L^k(V)\subset \pi_{m}^*(\mathcal{H}(\C^{m}))\]
			and so $\textnormal{orb}_L(V)$ cannot be dense in$(\mathcal{H}(\C^\N),\tau)$. Thus, no finite-dimensional subspace of $\mathcal{H}(\C^\N)$ is supercyclic for $L$. Hence $L$ is not $n$-supercyclic.	
		\end{enumerate}		
		
	\end{proof}
	
	\subsection{Li--Yorke chaos for convolution operators on $\mathcal{H}(\C^\N)$}

	The main result of this section is the following:
	
	\begin{theorem}\label{TLY}
		Every nontrivial convolution operator on $(\mathcal{H}(\C^\N),\tau)$ is Li--Yorke chaotic.
	\end{theorem}	
	
	To show this result we will prove a characterization of Li--Yorke chaos that involves the existence of semi-irregular vector. This characterization was obtained by Bernardes Jr \textit{et al} \cite{Li-Yorke} for operators on Fréchet spaces. We will generalize this fact for Hausdorff topological vector spaces.	The definition below is known for Fréchet spaces.
	
	Let $E$ be a Hausdorff topological vector space and let $T$ be an operator on $E$. A vector $x\in E$ is said to be a \emph{semi-irregular vector} for $T$ if the sequence $(T^nx)$ does not converge to zero, but it has a subsequence converging to zero. It is easy to see that  $(x,y)\in E\times E$ is a Li--Yorke pair for $T$ if and only if $x-y$ 
	is a semi-irregular vector for $T$.	
	
	As it was mentioned in \cite{Li-Yorke}, the notion of semi-irregularity makes sense only for infinite-dimensional spaces. An easy application of the Jordan form implies that there are no semi-irregular vectors for operators on finite-dimensional spaces.
	
	\begin{theorem}\label{TsI}
		Let $E$ be a Hausdorff topological vector space, and let $T$ be an operator on $E$. The following assertions are equivalent:
		\begin{enumerate}
			\item [(i)] $T$ is Li--Yorke chaotic.
			\item [(ii)] $T$ admits a Li--Yorke pair.
			\item [(iii)] $T$ admits a semi-irregular vector.
		\end{enumerate}	
	\end{theorem}
	\begin{proof}Since the implications 
		(i) $\Rightarrow$ (ii) $\Rightarrow$ (iii) are immediate we just need to show that
		(iii) $\Rightarrow$ (i).  Let $x$ be a semi-irregular vector for $T$. Then for every $\alpha, \lambda\in\C$, with $\alpha\neq\lambda$, the sequence $\{T^n(\alpha x-\lambda x)\}$ does not converge to zero, but it has a subsequence converging to zero. Hence span$\{x\}$ is an uncountable scrambled set for $T$ and thus $T$ is Li--Yorke chaotic.
	\end{proof}
	
	Now we are able to prove the main result of this section.
	
	\begin{proof}[Proof of Theorem \ref{TLY}] Let $L$ be a nontrivial convolution operator on $(\mathcal{H}(\C^\N),\tau)$. We will show that $L$ has a semi-irregular entire function. By Lemma \ref{GJD}(c) there exists a nontrivial convolution operator $L_n$ associated to $L$. Since $L_n$ is a nontrivial convolution operator on $\mathcal{H}(\C^n)$, it follows from the classical result of Godefroy and Shapiro that $L_n$ is hypercyclic. In particular there exists a semi-irregular function $f_n\in\mathcal{H}(\C^n)$ for $L_n$. If we set $f=f_n\circ\pi_n\in\mathcal{H}(\C^\N)$, then $f$ is a semi-irregular function for $L$. In fact, since
		\begin{equation}\label{w}
		L^kf=(L_n^kf_n)\circ\pi_n=\pi_n^*\left(L_n^kf_n\right), \ \ \textnormal{for every} \ k\in\N_0,
		\end{equation}	
		it follows immediately from \eqref{w} that the sequence $(L^kf)_{k=0}^\infty$ has a subsequence converging to zero. On the other hand, if $L^kf\to0$ in the topology of $(\mathcal{H}(\C^\N),\tau)$, then
		\[L_n^kf_n=J_n^*\circ\pi_n^*\left(L_n^kf_n\right)=J_n^*(L^kf)\to0 \ \textnormal{in} \ \mathcal{H}(\C^n)\]	
		when $k\to\infty$, but this contradicts the fact  that $f_n$ is a semi-irregular function for $L_n$. Therefore $(L^kf)_{k=0}^\infty$ does not  converge to zero, and so $f$ is a semi-irregular function for $L$. Applying Theorem \ref{TsI} we obtain the desired.
	\end{proof}

	\bigskip
	
	\noindent Vin\'{\i}cius V. F\'{a}varo: \newline Faculdade de Matem\'atica,
	Universidade Federal de Uberl\^andia, Uberl\^andia -MG, CEP 38400-902,
	Brazil, \newline email: vvfavaro@gmail.com
	
	\bigskip
	
	\noindent Blas M. Caraballo:\newline IMECC, UNICAMP, Rua S\'{e}rgio Buarque de Holanda, 651, Campinas - SP, CEP 13083-859, Brazil, \newline email: mbcaraballo07@gmail.com 
	
\end{document}